\documentclass[12pt]{article}
\usepackage{amssymb,amsmath,latexsym}
\usepackage{mathptmx}
\usepackage{graphicx}
\usepackage{subfig}
\usepackage{bibentry}
\usepackage{amsthm}

 \newtheorem{thm}{Theorem}[section]
 \newtheorem{prop}{Proposition}[section]
 \newtheorem{lem}{Lemma}[section]

 \newtheorem{exm}{Example}[section]

 \numberwithin{equation}{section}

\begin{document}
\title{ON GENERALIZED BOHR-SOMMERFELD QUANTIZATION RULES FOR OPERATORS  WITH PT SYMMETRY}
\author{A.IFA\,$^{1}$, N.M'HADHBI\,$^{1,2}$ \& M.ROULEUX\,$^{3}$}
\date{\today}
\maketitle
\centerline{${}^{1}$ Universit\'e de Tunis El-Manar, D\'{e}partement de  Math\'{e}matiques, 1091 Tunis, Tunisia}
\centerline{email: abdelwaheb.ifa@fsm.rnu.tn}\ \\ \ \\
\centerline{${}^2$ Department of Mathematics, College of Sciences and Arts}
\centerline{King Abdulaziz University, Rabigh Campus, P.O. Box 344,  Rabigh 21911, Saudi Arabia}
\centerline{email: nalmhadhbi@kau.edu.sa}\ \\ \ \\
\centerline{${}^3$ Aix Marseille Universit\'e, Centre de Physique  Th\'{e}orique, UMR 7332, 13288 Marseille, France}
\centerline{\& Universit\'{e} de Toulon, CNRS, CPT, UMR 7332, 83957 La  Garde, France}
\centerline{email: rouleux@univ-tln.fr}
\abstract{
We give Bohr-Sommerfeld rules corresponding to quasi-eigenvalues in  the pseudo-spectrum for a 1-D $h$-Pseudodifferential operator
verifying PT symmetry.
}
\section{Introduction and statement of the result}\label{s.introduction}
Let $p(x,\xi;h)$ be a smooth (possibly complex valued) Hamiltonian on  $T^{*}\mathbb{R}$, with the formal expansion
$p(x,\xi;h)\sim p_{0}(x,\xi)+h\,p_{1}(x,\xi)+h^{2}\,p_{2}(x,\xi)+\ldots$. Assume that  for some order function $m$, $p$ belongs to the space of symbols $S^0(m)$, with
\begin{equation}\label{p1}
S^{N}(m)=\{p\in C^{\infty}(T^{*}\,\mathbb{R}): \forall \alpha \in \mathbb{N}^{2}, \exists C_{\alpha}>0, \forall (x,\xi)\in T^{*}\,\mathbb{R};\; |\partial^{\alpha}_{(x,\xi)}p(x,\xi;h)|\leq C_{\alpha} h^{N} m(x,\xi)\}
\end{equation}
[for instance $m(x,\xi)=(1+|\xi|^{2})^{M}$], and that $p+i$ is  elliptic. This allows to take Weyl quantization
$P=p^{w}(x,h\,D_{x};h)$ of $p$
\begin{equation}\label{p2}
P\,u(x;h)=(2\,\pi\,h)^{-1}\,\int\int e^{\frac{i}{h}\,(x-y)\,\eta}\,p(\frac{x+y}{2},\eta;h)\,u(y)\,dy\,d\eta
\end{equation}
which we denote also by $P=Op^w(p)$, or $p=\sigma^w(P)$. We call as usual
$p_{0}$ the principal symbol, $p_{1}$ the sub-principal symbol, and  assume throughout that $p_{0}$ is real.

Fix some compact interval $I=[E_{-},E_{+}], E_{-}<E_{+}$ and assume,  following \cite{5} that there exists a topological ring $\mathcal{A}\subset T^{*}\mathbb{R}$ such that $\partial\,\mathcal{A}={A}_{+}\cup {A}_{-}$
with $A_{\pm}$ a connected component of $p_{0}^{-1}(E_{\pm})$. Assume  also that $p_{0}$
has no critical pont in $\mathcal{A}$, and $A_-$ is included in the  disc bounded by $A_+$
(if this is not the case, we can always change $p$ to $-p$). We define the microlocal well $W$ as the disc bounded by par $A_+$. This includes the case of the standard Hamiltonian  $p_{0}(x,\xi)=\xi^{2}+V(x)$, but allows also for more general  geometries.

For $E\in I$, let $\gamma_{E}\subset W$ be a compact embedded  Lagrangian manifold (periodic orbit) in the energy surface
$\{p_{0}(x,\xi)=E\}$, and for
$N=1,2,\cdots$, let $K_{h}^{N}(E)$ denote the microlocal kernel of $P-E$ up to  order $N$, i.e. the set of local solutions of
$(P-E)\,u=\mathcal{O}(h^{N+1})$ in the distributional sense,
microlocalized on $\gamma_{E}$. This is a smooth complex vector bundle over
$\pi_{x}(\gamma_{E})$, where $\pi_{x}:T^{*}\mathbb{R}\rightarrow \mathbb{R}$. Finding the set of $E=E(h;N)$ such that  $K_{h}^{N}(E)$ contains a global section, amounts to construct
normalized {\it quasi-modes} (QM) $(u(h;N),E(h;N))$ up to order $N$. In other words, the condition that determines the  sequence of
quasi-eigenvalues $E(h;N)=E_n(h;N)$ is that the corresponding  quasi-eigenfunction $u(h;N)=u_n(h;N)$ be single-valued. It is known as Bohr-Sommerfeld condition (BS for short). In the sequel, we drop index $N$ when unnecessary.

Assume that $P$ is self-adjoint, and $E_+<E_0=\displaystyle \liminf_{|x,\xi|\rightarrow \infty}\,p_{0}(x,\xi)$. Then BS determines
asymptotically all eigenvalues of $P$ in $I$, by the equation $\mathcal{S}_{h}(E_{n}(h))=2\,\pi\,n\,h$, to any order $N$. The
semi-classical action $\mathcal{S}_h(E)$ has asymptotics $\mathcal{S}_h(E)\sim S_{0}(E)+h\,S_{1}(E)+h^{2}\,S_{2}(E)+\ldots$. $S_0$ is the classical action
$\displaystyle \oint_{\gamma_{E}}\xi\,dx=\displaystyle \int_{\{p_{0}\leq E\}\cap W}d\xi\wedge \,dx$, $S_{1}(E)=\pi-\displaystyle \int_{\gamma_{E}}p_{1}((x(t),\xi(t))\,dt$
includes Maslov
correction and the subprincipal 1-form $p_1 dt$, where $t$ is the parameter in Hamilton equations.
Terms $S_0$ and $S_1$ are computed using Maslov canonical operator, or more specifically in the
present 1-D case, the monodromy operator (see \cite{10} and references therein). A more systematic
method (still in the 1-D case) is based on functional calculus for $h$-PDO's, in particular Moyal's
product, and uses a general formula due to \cite{15} (see also \cite{1} for earlier work). Thus, with
\begin{equation}\label{p3}
\Delta=\frac{\partial^{2}p_{0}}{\partial x^{2}}\;\frac{\partial^{2} p_{0}}{\partial \xi^{2}}-
(\,\frac{\partial^{2} p_{0}}{\partial x\,\partial \xi}\,)^{2}
\end{equation}
we have
\begin{equation}\label{p4}
S_{2}(E)=\frac{1}{24}\,\frac{d}{dE}\,\int_{\gamma_{E}}\Delta\,dt-
\int_{\gamma_{E}}p_{2}\,dt-\frac{1}{2}\,\frac{d}{dE}
\int_{\gamma_{E}}p_{1}^{2}\,dt
\end{equation}
This method was implemented in different ways by \cite{13}, \cite{5} and given later a diagrammatic
approach by \cite{4} and \cite{9}, providing an algorithm to compute all higher order
terms, in particular the 4th order term can be computed in a closed form without too much work.
Note that all $S_j(E)$ with $j\geq3$ odd vanish.

It is shown further in \cite{5}, using trace formulas, that BS gives actually all eigenvalues in $I$.
Note that this approach, in contrast with the method of the monodromy operator, assumes already
the existence of BS, and the problem is about the most efficient way of computing the $S_j$'s.

In the real analytic case, when $P=-h^{2}\,\Delta+V(x)$ is Schr\"{o}dinger operator, BS can be obtained
using the exact complex WKB method (see \cite{8}, \cite{6} and references therein); it consists first
in transforming the eigenvalue equation $-h^{2}\,u''(x)+V(x)\,u(x)=E\,u(x)$ into a Ricatti equation, and then
compute Jost function whose zeroes are precisely the eigenvalues of $P$.

Consider now a $h$-PDO $P$ (not necessarily self-adjoint) that satisfies PT symmetry i.e. $\mathrm{P}\mathcal{P}\mathcal{T}=\mathcal{P}\mathcal{T} \mathrm{P}$, where $\mathcal{P}\mathcal{T}=^{\checkmark}\mathcal{I}$, $^{\checkmark}$ is the parity operator
$^{\checkmark}u(x)=u(-x)$ and $\mathcal{I}$ the complex conjugation. At the level of Weyl symbol, this symmetry takes the form $p(-x,\xi;h)=\overline{p(x,\xi;h)}$. Such a property is sometimes considered in Physics as a natural substitute for self-adjointness. It is known
that finding quasi-modes is in no ways sufficient to get information about the spectrum of $P$, but
only about its pseudo-spectrum (see \cite{8}, \cite{B} and \cite{14} for more recent results). The pseudo-spectrum
is symmetric with respect to the real axis, and one expects generally to recover some real eigenvalues. We specialize further in the case where $P$ has a real principal symbol. Our main result
is the following:
\begin{thm}\label{Theo1}
Let $P$ as above enjoy PT symmetry, and $p_0$ be real. Then, for
at least $N=4$, there exists $b\in S^{0}(m)$ defined microlocally in $W$, such that $Q=B P B^{-1}, b=\sigma^{w}(B)$,
is formally self-adjoint (at least modulo an operator with symbol in $S^{N+1}(m))$. In particular, there is
a sequence of quasi-modes $(u_{n}(h),E_{n}(h))$ such that $(P-E_{n}(h))\,u_{n}(h)=\mathcal{O}(h^{N+1})$, with $E_{n}(h)\in I$, satisfying  $\mathcal{S}_{h}(E_{n}(h))=2\pi n h$, for an asymptotic series $\mathcal{S}_{h}(E)=\displaystyle \sum_{j=1}^{N}S_{j}(E)\,h^{j}+\mathcal{O}(h^{N+1})$ where $S_{j}\in \mathbb{R}$ are real. In particular, the pseudo-spectrum of $P$ lies within a distance $\mathcal{O}(h^{N+1})$ of $I$.
The coefficients $S_j(E)$ can be computed as in \cite{9} from the symbol of $Q$; thus
\begin{equation*}
S_{0}(E)=\displaystyle \oint_{\gamma_{E}} \xi(x)\,dx=\displaystyle \int\int_{\{p_{0}\leq E\}\cap W}d\xi\wedge dx
\end{equation*}
is the action integral,
\begin{equation*}
S_{1}(E)=\pi-\displaystyle \int_{\gamma_{E}} \mathrm{Re}(p_{1}(x(t),\xi(t))\,dt,
\end{equation*}
and
\begin{equation*}
S_{2}(E)=\displaystyle \frac{1}{24}\,\frac{d}{dE}\,\int_{\gamma_{E}}\Delta\,dt-\displaystyle
\int_{\gamma_{E}}\big(\mathrm{Re}(p_{2})-\displaystyle \frac{1}{2}\,\{\{\beta_{0},p_{0}\},\beta_{0}\}\big)\,dt-\displaystyle \frac{1}{2}\,
\frac{d}{dE}\,\int_{\gamma_{E}}(\mathrm{Re}(p_{1}))^{2}\,dt
\end{equation*}
with $\Delta$ as in (\ref{p3}), and
\begin{equation*}
\beta_{0}(x,\xi)=\displaystyle \oint_{\gamma_{E}}(1-\frac{s}{T(E)})\mathrm{Im}(p_{1})\circ \exp sH_{p_{0}}(x,\xi)ds
\end{equation*}
Denoting by $T(E)$ the period
of the flow on $\gamma_{E}$. Again, $S_{3}=0$, and $S_4(E)$ can be computed using (\cite{4},Formula (7.3)) and the
formula giving $\sigma^{w}(BPB^{-1})$ mod $\mathcal{O}(h^{5})$.
\end{thm}
Of course, we conjecture that Theorem \ref{Theo1} holds for all $N$.\ \\
\begin{exm} Consider the operator $Q(x,hD_{x})=(hD_{x})^{2}+p(x)\,hD_{x}+q(x)$ with smooth, real
coefficients. Then $Q$ can be mapped into $P(x,hD_{x})=(hD_{x})^{2}+q(x)-\frac{1}{4}\,(p(x))^{2}+i\,\frac{h}{2}\,p'(x)$ by the unitary
transformation $Q=BPB^{*}$, $Bv(x;h)=\exp(-\frac{i}{2h}\,\int^{x}p(t)\,dt)\,v(x)$. Assume $Q$ verifies PT symmetry,
i.e. $p$ and $q$ are even on $\mathbb{R}$, then the same holds for $P$. The microlocal well $W_{E}=\{(x,\xi)\in T^{*}(\mathbb{R}); \xi^{2}+q(x)-\frac{1}{4}\,(p(x))^{2}\leq E\}$ for $P$ projects onto the potential well $U_{E}=\{x\in \mathbb{R}; q(x)-\frac{1}{4}\,(p(x))^{2}\leq E\}$,
so Theorem \ref{Theo1} holds provided $V(x)=q(x)-\frac{1}{4}\,(p(x))^{2}$ has no critical point in $I$.

If $p$ and $q$ analytic, then the spectrum of $P$ is in fact real in $I$ and given by (exact) BS. In fact,
using also some of the technics elaborated in \cite{11}, \cite{3} showed that if $P=-h^{2}\,\Delta+V(x)+i\,\varepsilon\,W(x)$ is a small perturbation of the self-adjoint Schr\"{o}dinger operator $P_0=-h^{2}\,\Delta+V(x)$, then the semiclassical action is a real analytic function and the roots of BS are real eigenvalues of $P$. This implies
Theorem \ref{Theo1} by choosing $\varepsilon=h$, but the argument of \cite{3} heavily relies upon that particular form of $P$.
\end{exm}
\section{Proof of the Theorem}
Since we know $\mathcal{S}_{h}(E)$ for 1 self-adjoint operator (\cite{13},\cite{15}), it suffices to conjugate $P$ by an elliptic (but
non-unitary) $h$-PDO so that the resulting operator becomes formally self-adjoint up to order $N$. We proceed
in several steps.
\subsection{Birkhoff normal form (BNF)}
Let $\tilde{P}$ be self-adjoint as in (\ref{p3})  with (real) Weyl symbol $\tilde{p}\in S^{0}(m)$, and assume that its principal symbol $\tilde{p}_{0}=p_{0}$ has a periodic orbit $\gamma_{0}$ at non critical energy $E=0$. Then there exists a smooth canonical transformation $(s,\tau)\mapsto \kappa (s,\tau)=(x,\xi)$, $s\in [0,2 \pi]$, defined in a neighborhood of $\gamma_{0}$ and a smooth function $\tau\mapsto f_{0}(\tau)$, $f_{0}(0)=0$, $f'_{0}(0)\neq 0$ such that $p_{0}\circ \kappa(s,\tau)=f_{0}(\tau)$. Energy $E$ and momentum $\tau$ are related by the 1-to-1 transformation $E=f_{0}(\tau)$. This transformation can be quantized to take $\tilde{P}$ in its semi-classical BNF. Namely, there is a microlocally unitary $h$-FIO operator $U$ associated with $\kappa$, and a classical symbol $f(\tau;h)=f_{0}(\tau)+h\,f_{1}(\tau)+h^{2}\,f_{2}(\tau)+\cdots$, such that $U^{*}\tilde{P}U=f(hD_{s};h)$.
See e.g. \cite{2} at the level of the principal symbol, and \cite{12} for the full symbol; BNF turns
out to be convergent in the 1-D case.
In the canonical (action-angle) variables $(s,\tau)$, $s\in [0,2 \pi]$, the parity operator $^{\checkmark}:x\mapsto -x$ on the
real line takes the form $^{\checkmark}:s\mapsto \pi-s$ on the circle. Moreover, we can choose $U$ so that it commutes with PT symmetry: $U \mathcal{P} \mathcal{T}=\mathcal{P} \mathcal{T} U$.
\subsection{The homological equation}
We start with the following elementary result (see e.g. (\cite{14},p.93)):
\begin{lem}\label{L1}
Let $q$ and $p$ be real Hamiltonians. Then the equation
\begin{equation}\label{1}
q+\{\beta,p\}=0
\end{equation}
has a (global) real solution $\beta$ along $\gamma_{E}$ iff
\begin{equation}\label{2}
\oint_{\gamma_{E}}q\circ \exp tH_{p}(\rho)dt=0
\end{equation}
for any $\rho\in \gamma_{E}$. It is given by
\begin{equation}\label{3}
\beta(\rho)=-\oint_{\gamma_{E}}(1-\frac{t}{T(E)})q\circ \exp tH_{p}(\rho)dt
\end{equation}
\end{lem}
\begin{lem}\label{L2}
Assume $p=p_0$ as above and $q$ is odd with respect to $\mathcal{P} \mathcal{T}$; then (\ref{2}) holds.
\end{lem}
\begin{proof}
Using action-angle coordinates $(s,\tau)$ we have $p_{0}(s,\tau)=f_{0}(\tau)$, hence $H_{p_{0}}(t,\tau)=f'_{0}(\tau)\,\displaystyle \frac{\partial}{\partial t}$, where $f'_{0}(\tau)=\omega(E)$ is the energy dependent frequency.\ \\
For $\rho=(s,\tau)$, $\exp tH_{p_{0}}(\rho)=\phi_{t}(\rho)=(s+\omega(E)\,t,\tau)$. Then, using the periodicity of $q$
\begin{equation*}
\oint_{\gamma_{E}}q\circ \exp tH_{p}(\rho)dt=\int_{0}^{T(E)}q(s+\omega(E)t,\tau)dt=\frac{1}{\omega(E)}\int_{s}^{s+2\pi}q(s',\tau)ds'
\end{equation*}
which is 0 since $q(.,\tau)$ is odd as a function on the circle.
\end{proof}
\subsection{Reducing to a formally self-adjoint operator}
\begin{prop}
Let $p(x,\xi;h)\sim p_{0}(x,\xi)+h p_{1}(x,\xi)+h^{2} p_{2}(x,\xi)+\cdots \in S^{0}(m)$ satisfy PT symmetry with real $p_{0}$. Then at least for $N=4$, there exists $b\in S^{0}(m)$ elliptic such that Weyl symbol of $BPB^{-1}$, $\sigma^{W}(B)$, is real mod $O(h^{N+1})$. Moreover, $BPB^{-1}$ is again PT-symmetric (up to that order).
\end{prop}
\begin{proof}
To shorten the exposition, we content to the lower order accuracy $\mathcal{O}(h^{4})$.
First we carry BNF for the self-adjoint part $\tilde{P}=\displaystyle \frac{P+P^{*}}{2}$ of $P$, which has real Weyl symbol
and verifies PT symmetry. Since $U$ commutes with $\mathcal{P} \mathcal{T}$, the anti-self adjoint part $\displaystyle \frac{P-P^{*}}{2}$ also satisfies PT symmetry (but is not necessarily in BNF).

Check first the Proposition for $N=1$. Let $B_0$ have Weyl symbol $b_0$, which we write as $\sigma^{w}(B_{0})=b_0$.
Let $b_{0}=e^{\beta_{0}}$, with real $\beta_{0}$. By $h$-Pseudodifferential Calculus (i.e. Moyal product), $B_{0} P B_{0}^{-1}=[B_{0},P] B_{0}^{-1}+P$ has Weyl symbol
\begin{equation}\label{4}
\sigma^{W}\big(B_{0} P B_{0}^{-1}\big)=p-i\,h\{\beta_{0},p\}+\frac{h^{2}}{2}\,\big\{\{\beta_{0},p\},\beta_{0}\big\}+
i\,\frac{h^{3}}{8}\,R_{5}(\beta_{0},\alpha(p))+
\frac{h^{4}}{48}\,R_{8}(\beta_{0},\alpha(p))+O(h^{5})
\end{equation}
with
\begin{equation}\label{5}
\alpha(p)=\{\beta_{0},p\}
\end{equation}
and
\begin{align*}
\!\!\!R_{5}(\beta_{0},\alpha(p))&=\big((\partial_{\xi}\beta_{0})^{2}-\frac{\partial^{2}\beta_{0}}{\partial \xi^{2}}\big)\times\big(2\, \partial_{x}\alpha(p)\,\partial_{x} \beta_{0}+\alpha(p)\,\frac{\partial^{2}\beta_{0}}{\partial x^{2}}+\frac{\partial^{2} \alpha(p)}{\partial x^{2}}+\alpha(p) (\partial_{x} \beta_{0})^{2}\big)\ \\
&+\big((\partial_{x}\beta_{0})^{2}-\frac{\partial^{2}\beta_{0}}{\partial x^{2}}\big)\times \big(2\, \partial_{\xi}\alpha(p)\,\partial_{\xi} \beta_{0}+\alpha(p)\,\frac{\partial^{2}\beta_{0}}{\partial \xi^{2}}+\frac{\partial^{2} \alpha(p)}{\partial \xi^{2}}+\alpha(p) (\partial_{\xi} \beta_{0})^{2}\big)\ \\
&-2\big(\partial_{x}\beta_{0}\,\partial_{\xi}\beta_{0}-\frac{\partial^{2}\beta_{0}}{\partial x \partial \xi}\big)\times\big(\frac{\partial^{2}\alpha(p)}{\partial x \partial \xi}+\partial_{\xi}\alpha(p)\,\partial_{x}\beta_{0}+\partial_{x}\alpha(p)\,\partial_{\xi}\beta_{0}+
\alpha(p)\,\partial_{x}\beta_{0}\,\partial_{\xi}\beta_{0}+\alpha(p)\,\frac{\partial^{2}\beta_{0}}{\partial x \partial \xi}\big)
\end{align*}
\begin{align*}
R_{8}(\beta_{0},\alpha(p))&=F_{5}(\beta_{0},\alpha(p))\, \big(3\,\partial_{x}\beta_{0}\,\frac{\partial^{2}\beta_{0}}{\partial x^{2}}-
\frac{\partial^{3}\beta_{0}}{\partial x^{3}}-(\partial_{x} \beta_{0})^{3}
\big)\ \\
&-\tilde{F}_{5}(\beta_{0},\alpha(p))\,\big(3\,\partial_{\xi}\beta_{0}\,\frac{\partial^{2}\beta_{0}}{\partial \xi^{2}}-
\frac{\partial^{3}\beta_{0}}{\partial \xi^{3}}-(\partial_{\xi} \beta_{0})^{3}
\big)\ \\
&+3\,G_{5}(\beta_{0},\alpha(p))\,\big(2\,\partial_{\xi}\beta_{0}\,\frac{\partial^{2}\beta_{0}}{\partial x \partial \xi}-
\frac{\partial^{3}\beta_{0}}{\partial x \partial \xi^{2}}-\partial_{x} \beta_{0}\,(\partial_{\xi} \beta_{0})^{2}+\partial_{x} \beta_{0}\,\frac{\partial^{2}\beta_{0}}{\partial \xi^{2}}\big)\ \\
&-3\,\tilde{G}_{5}(\beta_{0},\alpha(p))\,\big(2\,\partial_{x}\beta_{0}\,\frac{\partial^{2}\beta_{0}}{\partial x \partial \xi}-
\frac{\partial^{3}\beta_{0}}{\partial x^{2} \partial \xi}-\partial_{\xi} \beta_{0}\,(\partial_{x} \beta_{0})^{2}+\partial_{\xi} \beta_{0}\,\frac{\partial^{2}\beta_{0}}{\partial x^{2}}\big)
\end{align*}
where
$$
F_{5}(\beta_{0},\alpha(p))=3\,\partial_{\xi}\alpha\,\frac{\partial^{2}\beta_{0}}{\partial \xi^{2}}+\alpha\,
\frac{\partial^{3}\beta_{0}}{\partial \xi^{3}}+3\,\partial_{\xi} \beta_{0}\,\frac{\partial^{2} \alpha}{\partial \xi^{2}}+\frac{\partial^{3}\alpha}{\partial \xi^{3}}+3\,\partial_{\xi} \alpha\,(\partial_{\xi} \beta_{0})^{2}+3\, \alpha\,\partial_{\xi} \beta_{0}\,\frac{\partial^{2}\beta_{0}}{\partial \xi^{2}}+\alpha\,(\partial_{\xi} \beta_{0})^{2}
$$
$$
\tilde{F}_{5}(\beta_{0},\alpha(p))=3\,\partial_{x}\alpha\,\frac{\partial^{2}\beta_{0}}{\partial x^{2}}+\alpha\,
\frac{\partial^{3}\beta_{0}}{\partial x^{3}}+3\,\partial_{x} \beta_{0}\,\frac{\partial^{2} \alpha}{\partial x^{2}}+\frac{\partial^{3}\alpha}{\partial x^{3}}+3\,\partial_{x} \alpha\,(\partial_{x} \beta_{0})^{2}+3\, \alpha\,\partial_{x} \beta_{0}\,\frac{\partial^{2}\beta_{0}}{\partial x^{2}}+\alpha\,(\partial_{x} \beta_{0})^{2}
$$
\begin{align*}\label{}
G_{5}(\beta_{0},\alpha(p))&=2\,\partial_{x}\beta_{0}\,\frac{\partial^{2}\alpha}{\partial x \partial \xi}+2\,\partial_{x}\alpha\,\frac{\partial^{2}\beta_{0}}{\partial x \partial \xi}+\partial_{\xi}\alpha\, \frac{\partial^{2}\beta_{0}}{\partial x^{2}}+\alpha\,\frac{\partial^{3}\beta_{0}}{\partial x^{2} \partial \xi}+\frac{\partial^{3}\alpha}{\partial x^{2} \partial \xi}\ \\
&+\partial_{\xi}\alpha\,(\partial_{x} \beta_{0})^{2}
+2 \alpha\,\partial_{x}\beta_{0}\,\frac{\partial^{2}\beta_{0}}{\partial x \partial \xi}
+\big(2\,\partial_{x}\alpha\,\partial_{x}\beta_{0}+\alpha\,\frac{\partial^{2}\beta_{0}}{\partial x^{2}}+\frac{\partial^{2}\alpha}{\partial x^{2}}+\alpha\,(\partial_{x}\beta_{0})^{2}\big)\,\partial_{\xi}\beta_{0}
\end{align*}
\begin{align*}
\tilde{G}_{5}(\beta_{0},\alpha(p))&=2\,\partial_{\xi}\beta_{0}\,\frac{\partial^{2}\alpha}{\partial x \partial \xi}+2\,\partial_{\xi}\alpha\,\frac{\partial^{2}\beta_{0}}{\partial x \partial \xi}+\partial_{x} \alpha\, \frac{\partial^{2}\beta_{0}}{\partial \xi^{2}}+\alpha\,\frac{\partial^{3}\beta_{0}}{\partial x \partial \xi^{2}}+\frac{\partial^{3}\alpha}{\partial x \partial \xi^{2}}+\partial_{x} \alpha\,(\partial_{\xi} \beta_{0})^{2}\ \\
&+2 \alpha\,\partial_{\xi} \beta_{0}\,\frac{\partial^{2}\beta_{0}}{\partial x \partial \xi}
+\big(2\,\partial_{\xi} \alpha\,\partial_{\xi}\beta_{0}+\alpha\,\frac{\partial^{2}\beta_{0}}{\partial \xi^{2}}+\frac{\partial^{2}\alpha}{\partial \xi^{2}}+\alpha\,(\partial_{\xi} \beta_{0})^{2}\big)\,\partial_{x} \beta_{0}
\end{align*}
Here $R_5(\beta_{0},\alpha(p_{0}))$ is a Hamilton-Jacobi polynomial with integer coefficients, polynomial in the derivatives of $\beta_{0}$ up to order 2, homogeneous of degree 5 (total degree in ($\partial_{x},\partial_{\xi}$)) when counting altogether products and derivatives; and similarly for $R_8(\beta_{0},\alpha(p_{0}))$.
Note that these Hamilton-Jacobi polynomials depend linearly on $\alpha(p)$. The first order term of the symbol is real iff
\begin{equation}\label{6}
\{\beta_{0},p_{0}\}=\mathrm{Im}(p_{1})
\end{equation}
and by Lemmas \ref{L1} and \ref{L2} this equation can be solved on $\gamma_{E}$, and
\begin{equation}\label{7}
\beta_{0}(s,\tau)=\displaystyle \int_{0}^{T(E)}(1-\frac{t}{T(E)})\,\mathrm{Im}(p_{1})(s+\omega(E)t,\tau)\,dt
\end{equation}
with
\begin{equation}\label{}
\omega(E)T(E)=2\,\pi
\end{equation}
We notice that $\beta_{0}(.,\tau)$ is an even function on the circle. So in (\ref{4}) we are left with
\begin{equation*}
\sigma^{W}(B_{0} P B_{0}^{-1})=p_{0}+h\,\mathrm{Re}(p_{1})+h^{2}\,\big(\mathrm{Re}(p_{2})-
\frac{1}{2}\,\big\{\{\beta_{0},p_{0}\},\beta_{0}\big\}\big)+\mathcal{O}(h^{3})
\end{equation*}
Chek now the  proposition for $N=2$. Let $B_1$ have Weyl symbol $\sigma^{W}(B_{1})=e^{h \beta_{1}}$ with $\beta_{1}$ real, and compute Weyl symbol of $B_{1} B_{0} P B_{0}^{-1} B_{1}^{-1}$.
Again by Moyal product, we get mod $\mathcal{O}(h^{3})$
\begin{equation}\label{9}
\!\!\!\sigma^{W}\big(B_{1} B_{0} P B_{0}^{-1} B_{1}^{-1}\big)
\equiv p_{0}+h \big(p_{1}-i\{\beta_{0},p_{0}\}\,)+h^{2}\,\big(p_{2}-i\{\beta_{0},p_{1}\}-i\{\beta_{1},p_{0}\}\\+
\frac{1}{2} \{\{\beta_{0},p_{0}\},\beta_{0}\}\big)
\end{equation}
The equation for $\beta_{1}$ reads
\begin{equation}\label{11}
\{\beta_{1},p_{0}\}=\mathrm{Im}(p_{2})-\{\beta_{0},\mathrm{Re}(p_{1})\}
\end{equation}
and we need to check the solvability condition (\ref{2}). It is fulfilled when $q=\mathrm{Im}(p_{2})$, since this is
an odd function on the circle; consider now $q=\{\beta_{0},\mathrm{Re}(p_{1})\}$, in action-angle coordinates we have
$\mathrm{Re}(p_{1})(t,\tau)=f_{1}(\tau)$, so $\{\beta_{0},\mathrm{Re}(p_{1})\}=-f'_{1}(\tau)\,\displaystyle \frac{\partial \beta_{0}}{\partial t}$. Since $\beta_{0}$ is $2\pi$-periodic
\begin{align*}
\int_{0}^{T(E)}\{\beta_{0},\mathrm{Re}(p_{1})\}(s+\omega(E)t,\tau)\,dt=-f'_{1}(\tau)\,(\beta_{0}(s+2 \pi,\tau)-\beta_{0}(s,\tau))=0
\end{align*}
so again the compatibility condition holds for $q=\{\beta_{0},\mathrm{Re}(p_{1})\}$, (\ref{11}) can be solved, and (\ref{9}) reduces to
\begin{equation}\label{12}
\sigma^{w}(B_{1} B_{0} P B_{0}^{-1} B_{1}^{-1})=p_{0}+h\,\mathrm{Re}(p_{1})+
h^{2}\,\big(\mathrm{Re}(p_{2})-\frac{1}{2}\,\big\{\{\beta_{0},p_{0}\},\beta_{0}\big\}
\,\big)+\mathcal{O}(h^{3})
\end{equation}
We notice that $\beta_{1}(.,\tau)$ is an even function on the circle.\ \\

Next we look for $B_{2}=Op^{W}(e^{h^{2} \beta_{2}})$, $\beta_{2}$ real so that $B_{2} B_{1} B_{0} P B_{0}^{-1} B_{1}^{-1} B_{2}^{-1}$ becomes self-adjoint up to $O(h^{3})$; the equation for $\beta_{2}$ reads
\begin{equation}\label{betha2}
\{\beta_{2},p_0\}=\mathrm{Im}(p_{3})-\sum_{k=0}^{1}\{\beta_{k},\mathrm{Re}(p_{2-k})\}+
\frac{1}{2}\,\{\{\beta_{0},\mathrm{Im}(p_{1})\},\beta_{0}\}+
\frac{1}{8}\,R_5(\beta_{0},\alpha(p_{0}))
\end{equation}
We need to check the compatibility condition for solving (\ref{betha2}), by the previous work it suffices to
consider $q=\{\{\beta_{0},\mathrm{Im}(p_{1})\},\beta_{0}\}$, and $q=R_5(\beta_{0},\alpha(p_{0}))$. Using again action-angle variables, we have: \begin{align*}
\!\!\{\{\beta_{0},\mathrm{Im}(p_{1})\},\beta_{0}\}&=
3 f''_{0}(\tau)\,(\frac{\partial \beta_{0}}{\partial t})^{2}\,\frac{\partial^{2} \beta_{0}}{\partial \tau \partial t}+f'_{0}(\tau)\,\frac{\partial \beta_{0}}{\partial t}\,(\frac{\partial^{2} \beta_{0}}{\partial \tau \partial t})^{2}
+f'_{0}(\tau)\,(\frac{\partial \beta_{0}}{\partial t})^{2}\,\frac{\partial^{3} \beta_{0}}{\partial \tau^{2} \partial t}+
f'''_{0}(\tau)\,(\frac{\partial \beta_{0}}{\partial t})^{3}\ \\
&-f'_{0}(\tau) \frac{\partial^{2} \beta_{0}}{\partial \tau^{2}} \frac{\partial \beta_{0}}{\partial t} \frac{\partial^{2} \beta_{0}}{\partial t^{2}}-2 f'_{0}(\tau) \frac{\partial \beta_{0}}{\partial \tau} \frac{\partial \beta_{0}}{\partial t} \frac{\partial^{3} \beta_{0}}{\partial \tau \partial t^{2}}
-3 f''_{0}(\tau) \frac{\partial \beta_{0}}{\partial \tau} \frac{\partial \beta_{0}}{\partial t} \frac{\partial^{2} \beta_{0}}{\partial t^{2}}+f'_{0}(\tau) (\frac{\partial \beta_{0}}{\partial \tau})^{2}\,\frac{\partial^{3} \beta_{0}}{\partial t^{3}}
\end{align*}
and the compatibility condition is fulfilled for that term since all functions to be integrated on the
circle are odd.

For $R_5(\beta_{0},\alpha(p_{0}))$ we proceed similarly. We check again that $\beta_{2}$ is an even function on the circle (i.e. under the transformation
$t\mapsto \pi-t$).\ \\
Once we know $\beta_{2}$, we compute $\sigma^{w}(B_{2} B_{1} B_{0} P B_{0}^{-1} B_{1}^{-1} B_{2}^{-1})$ mod $\mathcal{O}(h^{4})$, this gives:
\begin{align*}
\sigma^{W}(B_{2} B_{1} B_{0} P B_{0}^{-1} B_{1}^{-1} B_{2}^{-1} )&\equiv p_{0}+h\,\mathrm{Re}(p_{1})+h^{2}\,\big(\mathrm{Re}(p_{2})- \frac{1}{2}\,\big\{\{\beta_{0},p_{0}\},\beta_{0}\big\}
\,\big)\ \\
&+h^{3}\,\big(\mathrm{Re}(p_{3})- \big\{\{\beta_{1},p_{0}\},\beta_{0}\big\}-\frac{1}{2}\,
\big\{\{\beta_{0},\mathrm{Re}(p_{1})\},\beta_{0}\}\big)
\end{align*}
Let now $B_3$ have Weyl symbol $\sigma^{W}(B_{3})=e^{h^{3} \beta_{3}}$ with $\beta_{3}$ real, and compute Weyl symbol of $B_{3} B_{2} B_{1} B_{0} P B_{0}^{-1} B_{1}^{-1} B_{2}^{-1} B_{3}^{-1}$.
Again by Moyal product, we get mod $\mathcal{O}(h^{5})$
\begin{align*}
\!\!\!\!\!\!\!\sigma^{w}\big(B_{3} B_{2} B_{1} B_{0} P B_{0}^{-1} B_{1}^{-1} B_{2}^{-1} B_{3}^{-1}\big)&\equiv p-ih \sum_{j=0}^{3}h^{j} \{\beta_{j},p\}-h^{3}\,
\sum_{j=0}^{1}h^{j}\big\{\beta_{j+1},\{\beta_{0},p\}\big\}
+\frac{h^{2}}{2} \sum_{j=0}^{1}h^{2j}\big\{\{\beta_{j},p\},\beta_{j}\big\}\ \\
&+i\,\frac{h^{3}}{8}\,R_{5}(\beta_{0},\alpha(p))
+\frac{h^{4}}{48}\,R_{8}(\beta_{0},\alpha(p))
-i\,\frac{h^{4}}{2}\,\big\{\beta_{1},\big\{\{\beta_{0},p\},\beta_{0}\big\}
\big\}
\end{align*}
The equation for $\beta_3$ now reads
$$
\{\beta_{3},p_{0}\}=\mathrm{Im}(p_{4})- \sum_{k=0}^{2}\{\beta_{k},\mathrm{Re}(p_{3-k})\}+
\frac{1}{2}\,\{\{\beta_{0},\mathrm{Im}(p_{2})\},\beta_{0}\}
-\frac{1}{2}\,\big\{\big\{ \{\beta_{0},p_{0}\},\beta_{0}\big\},\beta_{1}\big\}
+\frac{1}{8}\,R_5\big(\,\beta_{0},\alpha\big(\mathrm{Re}(p_{1})\big)\big)\label{}
$$
The compatibility condition is fulfilled for $\mathrm{Im}(p_{4})$ and $\{\beta_{k},\text{Re}(p_{3-k})\}, 0\leq k\leq 2$ as before, for the double Poisson bracket $\{\{\beta_{0},\mathrm{Im}(p_{2})\},\beta_{0}\}$, the triple Poisson bracket $\big\{\big\{\{\beta_{0},p_{0}\},\beta_{0}\big\},\beta_{1}\big\}$ and $R_5\big(\,\beta_{0},\alpha\big(\mathrm{Re}(p_{1})\big)\,\big)$ we proceed similarly. We deduce $S_3(E)=0$.

Once we know $\beta_{3}$, we compute $\sigma^{w}(B_{3} B_{2} B_{1} B_{0} P B_{0}^{-1} B_{1}^{-1} B_{2}^{-1} B_{3}^{-1})$ mod $\mathcal{O}(h^{5})$
this gives:
\begin{align*}
\!\!\!\!\!\!\sigma^{w}(B_{3} B_{2} B_{1} B_{0} P B_{0}^{-1} B_{1}^{-1} B_{2}^{-1} B_{3}^{-1})&= p_{0}+h\,\mathrm{Re}(p_{1})+h^{2}\,\big(\mathrm{Re}(p_{2})-
\frac{1}{2}\,\{\{\beta_{0},p_{0}\},\beta_{0}\}\big)\ \\
&+h^{3}\,\big(\mathrm{Re}(p_{3})- \big\{\{\beta_{1},p_{0}\},\beta_{0}\big\}
-\frac{1}{2}\,\big\{\{\beta_{0},\mathrm{Re}(p_{1})\},\beta_{0}\big\}\big)\ \\
&+h^{4} \big(\mathrm{Re}(p_{4})+\sum_{k=0}^{1}\{\beta_{k},\mathrm{Im}(p_{3-k})\}+
\frac{1}{2} \{\{\beta_{0},\mathrm{Re}(p_{2})\},\beta_{0}\}+
\frac{1}{2} \{\{\beta_{1},p_{0}\},\beta_{1}\}\big)\ \\
&+h^{4}\,\big(\big\{\{\beta_{0},\mathrm{Re}(p_{1})\},\beta_{1}\big\}- \frac{1}{8}\,R_5\big(\beta_{0},\alpha\big(\mathrm{Re}(p_{1})\big)\big)+
\frac{1}{48}\,R_8\big(\beta_{0},\alpha(p_{0})\big)\big)
\end{align*}
\end{proof}
To prove the Theorem, we observe eventually that the knowledge of the symbol of $BPB^{-1}$ mod $\mathcal{O}(h^{5})$
(\cite{9},Formula(7.3)) is sufficient to compute $S_4(E)$ (although this formula was derived in the
particular case where the symbol of $P$ contains only $p_0$).
\section{Extension to operators with periodic coefficients}
As in \cite{5}, we replace $T^{*}\mathbb{R}$ by $T^{*}\mathbb{S}^{1}$, and the hypothesis on $P$ by the following:
\begin{description}
\item[$\bullet$] there is a topological ring $\mathcal{A}$, homotopic to the zero section of $T^{*}\mathbb{S}^{1}$, such that $\partial \mathcal{A}=\mathcal{A}_{-}\cup \mathcal{A}_{+}$ with $\mathcal{A}_{\pm}$ a
connected component of $p_{0}^{-1}(E_{\pm})$.
\item[$\bullet$]$p_0$ has no critical points in $\mathcal{A}$.
\item[$\bullet$]$\mathcal{A}$ is "below" $\mathcal{A}_{+}$.
\end{description}
Then Theorem \ref{Theo1} holds; the only change is that $S_{1}(E)=0$. Again we reduce $P$ to $f(hD_{t};h)$ as an operator on $\mathbb{S}^{1}$.\ \\
\textit{Acknowledgments}:
The second author (N.Mhadhbi) acknowledges with thanks the Deanship
of Scientific Research DSR, King Abdulaziz University (Jeddah) for its support. The third author
(M.Rouleux) cheerfully thanks S.Dobrokhotov, for his kind hospitality at Ishlinskiy Institute for Problems
of Mechanics (Moscow).

\end{document}